\documentclass[10pt]{article}
\usepackage{amsthm}
\usepackage{amsfonts}
\usepackage{amsmath}
\usepackage{amssymb}
\usepackage{graphicx}
\usepackage{epsf}
\usepackage{float}
\theoremstyle{plain}
\newtheorem{theorem}{Theorem}[section]

\newtheorem{proposition}[theorem]{Proposition}
\newtheorem{conjecture}[theorem]{Conjecture}

\newtheorem*{theoremmain}{Theorem \ref{thm-main}}
\theoremstyle{definition}

\theoremstyle{remark}

\advance\belowcaptionskip by -6pt
\def\emph#1{{\em #1\/}}

\def\S{{\fam0 S}}

\let\connsum=\#

\def\qmod#1{\allowbreak\kern0.5em\hbox{\rm(mod $#1$)}}

\def\set#1{\{#1\}}
\def\ee{\-} 
\def\vv(#1,#2){(#1,#2)} 
\def\cyc#1{(\!(#1)\!)}
\def\mZ{{\mathbb{Z}}}
\let\cart\times
\def\tline#1{\hbox to \hsize{#1}}
\def\er{r}
\def\ec{c}
\def\S{\mathcal{S}}
\long\def\ignore#1{}

\newbox\cobox
\def\crossout#1{%
 \setbox\cobox\hbox{#1}%
 \hbox to 0pt{\vrule height0.6ex depth-0.5ex width\wd\cobox\hss}%
 \unhbox\cobox}
\def\incgraphics#1{{\hbox{($#1$ goes here)}}}

\begin{document}


\title{\textbf{Criticality of counterexamples to toroidal
	edge-hamiltonicity%
	\thanks{The United States Government is authorized to reproduce
	and distribute reprints notwithstanding any copyright notation
	herein.}}}
\author{%
 M. N. Ellingham \\
 {\sl Department of Mathematics, 1326 Stevenson Center} \\
 {\sl Vanderbilt University, Nashville, TN 37240, U.S.A.} \\
 {\tt mark.ellingham@vanderbilt.edu}\\
 \and
 Emily A. Marshall \\
 {\sl Department of Mathematics, 1326 Stevenson Center} \\
 {\sl Vanderbilt University, Nashville, TN 37240, U.S.A.} \\
 {\tt emily.a.marshall@vanderbilt.edu}\\
}
\date{4 December 2013}
\maketitle 

\begin{abstract}
 A well-known conjecture of Gr\"unbaum and Nash-Williams proposes that
$4$-connected toroidal graphs are hamiltonian.  The corresponding
results for $4$-connected planar and projective-planar graphs were
proved by Tutte and by Thomas and Yu, respectively, using induction
arguments that proved a stronger result, that every edge is on a
hamilton cycle.  However, this stronger property does not hold for
$4$-connected toroidal graphs: Thomassen constructed counterexamples. 
Thus, the standard inductive approach will not work for the torus.  One
possible way to modify it is by characterizing the situations where some
edge is not on a hamilton cycle.  We provide a contribution in this
direction, by showing that the obvious generalizations of Thomassen's
counterexamples are critical in a certain sense.
 \end{abstract}

\section{Introduction}
\label{sec-intro}

 The study of hamilton cycles for graphs on surfaces was begun in 1931
by Whitney \cite{Wh31}, who showed that $4$-connected planar
triangulations are hamiltonian.  Tutte \cite{Tu56, Tu77} later
generalized this to all $4$-connected planar graphs.  Thomassen
\cite{Th83} (with a minor correction by Chiba and Nishizeki
\cite{ChNi86}) further extended this by showing that $4$-connected
planar graphs are hamilton-connected.
 Thomas and Yu \cite{ThYu94} showed that $4$-connected projective-planar
graphs are hamiltonian.

 In this paper we will be concerned with the following conjecture.

 \begin{conjecture}[{\rm Gr\"{u}nbaum \cite{Gr70} and Nash-Williams
\cite{NW73}}]\label{conj-grnw}
 Every $4$-connected toroidal graph is hamiltonian.
 \end{conjecture}

 \noindent
 A number of partial results are known.
 Altshuler \cite{Al72} showed that $6$-connected toroidal graphs, which
are $6$-regular triangulations with a grid structure, are
hamiltonian.
 In the same paper he also showed that $4$-connected toroidal
quadrangulations, which are $4$-regular with a grid structure, are
hamiltonian.
 Brunet and Richter \cite{BrRi95} proved that $5$-connected toroidal
triangulations are hamiltonian, and this was generalized by
 Thomas and Yu \cite{ThYu97} to all $5$-connected toroidal graphs.
 Thomas, Yu and Zang \cite{ThYuZa05} showed that every $4$-connected
toroidal graph has a hamilton path.
 Recently some special classes of toroidal graphs, including
$4$-connected toroidal graphs with toughness exactly $1$, were shown to
be hamiltonian by Nakamoto and Ozeki and by those two authors with
Fujisawa \cite{FuNaOz13, NaOz12}.
 However, a complete proof of Conjecture \ref{conj-grnw} still seems a
long way off.

 One reason Conjecture \ref{conj-grnw} seems difficult to prove is that
the standard inductive approach used for the plane and projective plane
cannot be extended to the torus.  The results for $4$-connected planar
and projective-planar graphs in \cite{ThYu94, Tu56, Tu77} are
essentially proved by strengthening the result in two ways, to enable
induction to be used.
 The first strengthening is to look for what are known as \emph{Tutte
cycles} instead of hamilton cycles, in $2$-connected graphs instead of
$4$-connected graphs.  In the $4$-connected case a Tutte cycle must be a
hamilton cycle.
 Some additional control over the Tutte cycles is needed, and so the
second strengthening is to make sure that the Tutte cycle can use any
given edge on a designated `boundary' of the graph.  For
$4$-connected planar or projective-planar graphs, therefore, this means
that they are not just hamiltonian but \emph{edge-hamiltonian}: every
edge has a hamilton cycle through it.
 $4$-connected toroidal graphs, however, are not in general edge-hamiltonian, and
so the same type of inductive arguments fail.

 Examples of non-edge-hamiltonian $4$-connected toroidal graphs were
given by Thomassen \cite{Th83}.  He observed that the cartesian product
of two even cycles yields a bipartite $4$-connected quadrangulation of
the torus, and if a \emph{diagonal} (an edge between opposite vertices)
is added in any quadrangle, then that diagonal cannot be in a hamilton
cycle.  This construction is easily generalized.  Take any bipartite
$4$-connected toroidal quadrangulation $Q$,  say with a bipartition into
black and white vertices.
 As mentioned earlier, $Q$ has a grid structure, which we discuss in
more detail later.  It also has equally many black and white vertices.
 In each quadrangle we can add either a \emph{black-black} or
\emph{white-white} diagonal, specifying the color of its ends.
 For any nonempty subset of the quadrangles, add a black-black diagonal
across each quadrangle.  Then the resulting $4$-connected toroidal
graph does not have a hamilton cycle through any of the added diagonals.
 We will call these \emph{grid-type} examples.

 Even more generally, we can take a bipartite quadrangulation of the
torus in which there are equally many black and white vertices, and all
white vertices have degree $4$.  There may be black vertices of degree
$2$ or $3$, so the connectivity may be less than $4$.  However, it may
be possible to make the graph $4$-connected by adding black-black
diagonals in some quadrangles.  The added diagonals will again not be on
a hamilton cycle.  In Figure \ref{fig-nongrid} the solid edges form a
quadrangulation of the torus (represented in the usual way, as a
rectangle with opposite sides identified) that is only $2$-connected. 
The addition of the four diagonals (dashed edges) makes it
$4$-connected, but the diagonals are not on any hamilton cycle.
 These examples, however, are much harder to characterize than the
grid-type examples.

\begin{figure}
 \tline{%
	\hfill
	\epsfbox{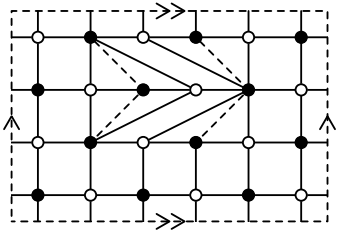}%
	\hfill
 }
 \begin{center}
 \caption{Non-grid example\label{fig-nongrid}}
 \end{center}
\end{figure}


 %

 Because of these examples, the inductive approach used for planar and
projective-planar graphs cannot be used for the torus without
modification.  A suitable modification might be to prove a result saying
that every $2$-connected toroidal graph has a Tutte cycle through any
boundary edge, except when a specific structure resulting from a
bipartite subgraph occurs.  Before trying to prove such a result,
however, it seems sensible to obtain some evidence as to whether the
problem (lack of edge-hamiltonicity) disappears when we depart even
slightly from the bipartite situation. 
 In this paper we address this by showing that the grid-type examples
are critical, in the sense that adding even one white-white diagonal, in
addition to the already added black-black diagonals, restores
edge-hamiltonicity.  Our main theorem is therefore as follows.

 \begin{theorem}\label{thm-main}
 \global\def\MainTheoremText{%
 Let $G$ be a 4-connected, 4-regular, bipartite simple graph on the
torus with partition sets of white and black vertices. 
 If we add a nonempty set $E_1$ of one or more black-black diagonals to
$G$, then no element of $E_1$ lies on a hamilton cycle in $G \cup E_1$.
 However, if we add one further white-white diagonal $e_2$ in a
quadrangle of $G \cup E_1$ then each edge of $G \cup E_1 \cup \set{e_2}$
lies on a hamilton cycle of that graph.
 }\MainTheoremText
 \end{theorem}

 The proof of this result makes up Section \ref{sec-mainproof}, and in
Section \ref{sec-concl} we give some concluding remarks.

 \section{Proof of the main result}\label{sec-mainproof}

\begin{figure}
 \tline{%
	\epsfbox{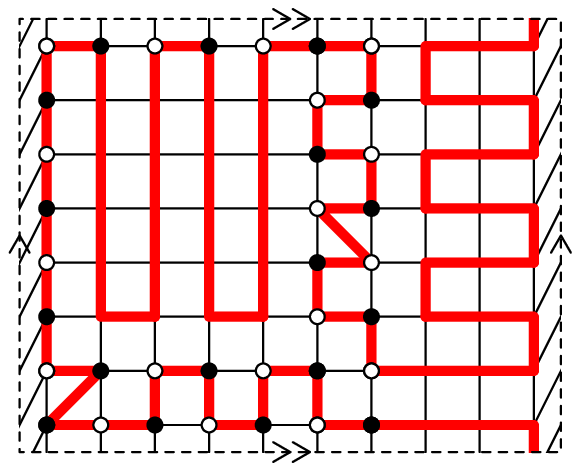}%
	\hfill
	\epsfbox{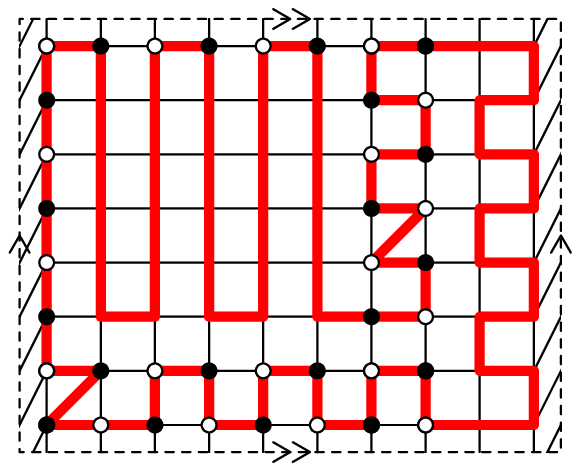}%
	\hfill
	\epsfbox{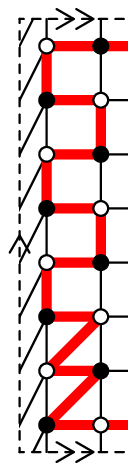}%
	\hfill
	\epsfbox{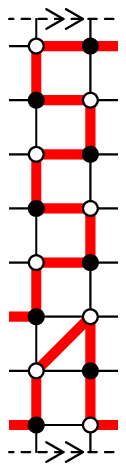}%
 }
 \tline{%
	\hbox to 60.5truemm{\hfil(a) $\er \ge 2$, $\ec$ even\hss}%
	\hfill
	\hbox to 60.5truemm{\hfil(b) $\er \ge 3$, $\ec$ odd\hss}%
	\hfill
	\hbox to 11truemm{\hfil(c)\hss}%
	\hfill
	\hbox to 11truemm{\hfil(d)\hss}%
 }
 \begin{center}
 \caption{Case 1.1\label{fig-case1.1}}
 \end{center}
\end{figure}

 Most of the proof of Theorem \ref{thm-main} is accomplished by the
following proposition.

 \begin{proposition}\label{prop-main}
 Let $G$ be a 4-connected, 4-regular, bipartite simple graph on the
torus with partition sets of white and black vertices.  Suppose we add a
black-black diagonal $e_1$ in one quadrangle of $G$, and a white-white
diagonal $e_2$ in a different quadrangle.  Then the resulting graph has
a hamilton cycle that uses both $e_1$ and $e_2$.
 \end{proposition}

 \begin{proof}
By Euler's formula, we know that all 4-regular, bipartite graphs on
the torus are quadrangulations.  
 As is well known \cite{Al72, NaNe00, Th91} $4$-regular
quadrangulations of the torus (bipartite or not) can be described (not
necessarily uniquely) by three integer parameters $m \ge 1$ (width), $n \ge 1$
(height) and $q$ (shift).
 To construct the quadrangulation we will denote $Q(m,n;q)$, take an
$m$-vertex path $P_m$ with vertex set $\mZ_m = \set{0, 1, 2, \ldots,
m-1}$ and an $n$-vertex cycle $C_n$ with vertex set $\mZ_n = \set{0, 1,
2, \ldots, n-1}$ (vertices labeled in the obvious order in each case).
 Representing the torus as a rectangle with opposite sides identified,
embed the cartesian product $P_m \cart C_n$ with the copies of $P_m$
horizontal and the copies of $C_n$ vertical.  Vertices are identified by
ordered pairs $(i,j)$ with $i \in \mZ_m$ and $j \in \mZ_n$, and we
specify edges and paths by concatenated ordered pairs.  We place vertex
$\vv(0,0)$ at  bottom left, and $\vv(m-1,n-1)$ at top right.
 In the cylindrical face between cycles $\set{m-1} \cart C_n$ and
$\set{0} \cart C_n$ add edges $\vv(m-1, j)\ee\vv(0, j+q)$ for $j \in \mZ_n$
 (so only the value of $q$ modulo $m$ matters).
 For example, Figure \ref{fig-case1.1}(a) and (b) show $Q(10,8;2)$ with
additional diagonals $e_1, e_2$.


 Each $Q(m, n; q)$ has an automorphism $U$ (translation up) which maps
every $\vv(i,j) \mapsto \vv(i, j+1)$, and an automorphism $R$
(translation right) which maps $\vv(i,j) \mapsto \vv(i+1,j)$ for $i \ne
m-1$ and $\vv(m-1, j) \mapsto \vv(0, j+q)$.
 There are also isomorphisms $F_1, F_2$ (reflections) from $Q(m, n; q)$
to $Q(m, n; -q)$: $F_1$ maps $\vv(i,j) \mapsto \vv(i,-j)$, and $F_2$
maps $\vv(i,j) \mapsto \vv(m-1-i, j)$.

 Now $G=Q(m,n;q)$ for some $m$, $n$ and $q$.  Since $G$ is bipartite,
$n$ must be even.  Since $G$ is simple, $n \ge 4$, and there are
restrictions on $q$ if $m = 1$ or $2$, which we discuss later.
 In the toroidal embedding of $G$, number the columns of faces $0, 1, 2,
\ldots, m-1$, so that column $i$ consists of faces between $\set{i-1}
\cart C_n$ and $\set{i} \cart C_n$.
 Similarly, number the rows of faces $0, 1, 2, \ldots, n-1$ so that row
$j$ consists of faces between $P_m \cart \set{j-1}$ and $P_m \cart
\set{j}$ (faces in column $0$ do not have a row number).

 \smallskip
 \noindent\textbf{Case 1.}  Suppose that $m \ge 3$, or that $m=2$ and
$e_1$ and $e_2$ are in the same column.  By applying a suitable power of
$R$ we can assume that neither $e_1$ nor $e_2$ is in column $0$, and at
least one of them is in column $1$.  Without loss of generality suppose
$e_1$ is in column $1$.  By applying a suitable power of $U$ we can make
one end of $e_1$ be $\vv(0,0)$.  Then, applying $F_1$ if necessary
(which negates $q$, but the value of $q$ will not matter in Case 1), we
can assume that $e_1 = \vv(0,0)\ee\vv(1,1)$.

 Now let $\ec$ and $\er$ be respectively the column and row of the face
for which $e_2$ is a diagonal.  We have ensured that $\ec \ne 0$, but
possibly $\er = 0$.

 \smallskip
 \noindent\textbf{ Case 1.1.} Suppose $\er \ge 2$.  If $\ec$ is even, then
we can find a hamilton cycle through $e_1$ as shown in Figure
\ref{fig-case1.1}(a); this works even if $\er=2$ or $\ec=2$ or both, and
regardless of whether $\er$ is odd or even.
 If $\ec=m-1$ then we replace the horizontal zigzag on the right which
joins $\vv(\ec,0)$ to $\vv(\ec,1)$ by the
single edge $\vv(m-1,0)\ee\vv(m-1,1)$.

 If $\ec$ is odd and $\er \ge 3$ then
we can find a hamilton cycle through $e_1$ and $e_2$ as shown in Figure
\ref{fig-case1.1}(b); this works even if $\ec=1$, and regardless of
whether $\er$ is odd or even.
 If $\ec=1$ and $\er=2$ we modify column $1$ as shown in Figure
\ref{fig-case1.1}(c).
 If $\ec \ge 3$ is odd and $\er = 2$ we modify column $\ec$ as shown in
Figure \ref{fig-case1.1}(d).
 If $\ec=m-1$ in any of these cases then we replace the horizontal
zigzag on the right which joins $\vv(\ec,0)$ to $\vv(\ec,n-1)$ by the edge
$\vv(m-1,0)\ee\vv(m-1,n-1)$.

\begin{figure}
 \tline{%
	\hfill
	\epsfbox{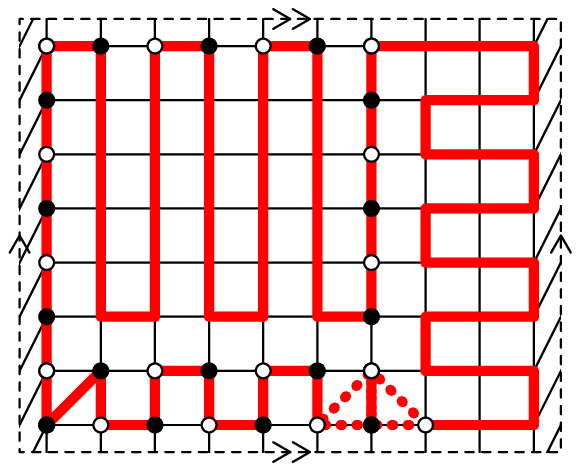}%
	\hfill
	\epsfbox{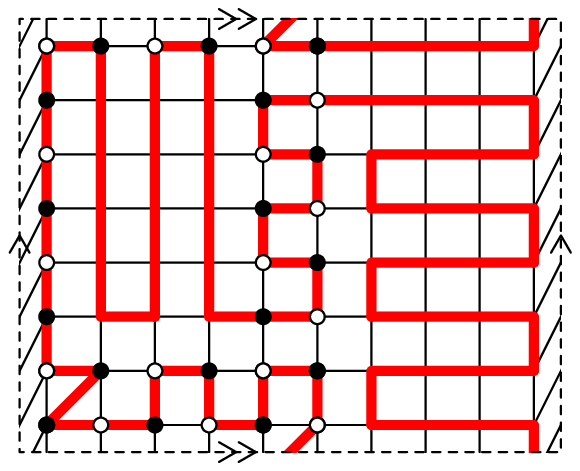}%
	\hfill
 }
 \tline{%
	\hfill
	\hbox to 60.5truemm{\hfil(a) $\er=1$\hss}%
	\hfill
	\hbox to 60.5truemm{\hfil(b) $\er=0$, $\ec$ odd\hss}%
	\hfill
 }
 \begin{center}
 \caption{Cases 1.2 and 1.3\label{fig-cases1.2-3}}
 \end{center}
\end{figure}

 \smallskip
 \noindent\textbf{ Case 1.2.} Suppose $\er = 1$.  Then $\ec \ge 2$ and $m
\ge 3$.
 We have a hamilton cycle through $e_1$ and $e_2$
as shown in Figure \ref{fig-cases1.2-3}(a).
 If $\ec=2k$ is even then we use the path
$\vv(2k-1,0)\ee\vv(2k,1)\ee\vv(2k,0)\ee\vv(2k+1,0)$,
 and if $\ec=2k+1$ is odd then we use the path
$\vv(2k-1,0)\ee\vv(2k,0)\ee\vv(2k,1)\ee\vv(2k+1,0)$.
 This works even if
$\ec=2$.  If $\ec$ is even and $\ec = m-1$ then
 we replace the horizontal zigzag which joins
$\vv(\ec,0)$ to $\vv(\ec,n-1)$ by the edge $\vv(m-1,0)\ee\vv(m-1,n-1)$.
 If $\ec$ is odd then the original construction works even if
$\ec=m-1$.

 \smallskip
 \noindent\textbf{ Case 1.3.} Suppose $\er = 0$.
 If $\ec$ is even then $e_2$ has the form $\vv(\ec-1,0)\ee\vv(\ec,n-1)$.  We
apply $U$ then $R^{m-1-\ec}$ then $F_2$, which move $e_2$ to
$\vv(\ec-1,1)\ee\vv(\ec,0)$ then to $\vv(m-2,1)\ee\vv(m-1,0)$ then to
$\vv(1,1)\ee\vv(0,0)=e_2'$.  These move $e_1$ to $\vv(0,1)\ee\vv(1,2)$ then to
$\vv(m-1-\ec,1)\ee\vv(m-\ec,2)$ then to $\vv(\ec,1)\ee\vv(\ec-1,2)=e_1'$.  Now
$e_1'$ is not in column $0$ or row $0$ so we can apply an earlier case
to $e_2'$ and $e_1'$, replacing $e_1$ and $e_2$ respectively.

 So $\ec$ is odd.  Then we can find a hamilton
cycle through $e_1$ and $e_2$ as shown in Figure
\ref{fig-cases1.2-3}(b).  If $\ec=m-1$ then we replace the horizontal
zigzag on the right by the edge $\vv(m-1,n-1)\ee\vv(m-1,n-2)$.  This
works even if $\ec=1$.

\begin{figure}
 \tline{%
	\kern15pt%
	\epsfbox{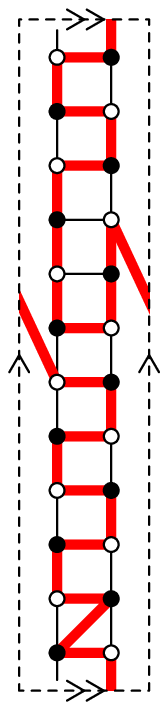}%
	\hfill
	\epsfbox{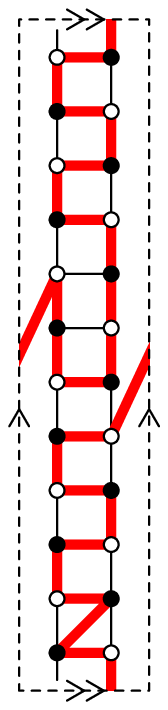}%
	\hfill
	\epsfbox{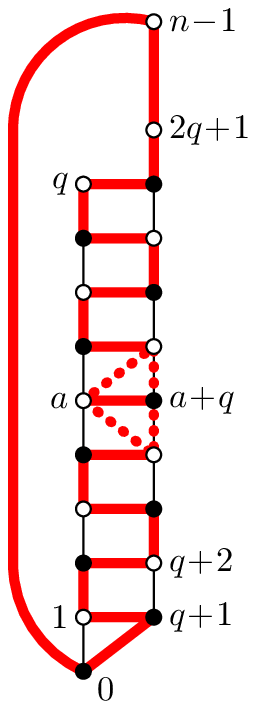}%
	\hfill
	\epsfbox{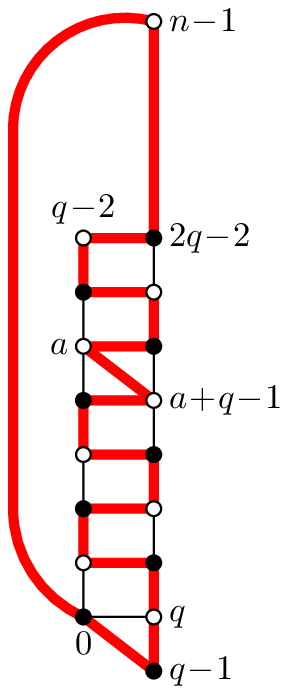}%
	\kern15pt%
 }
 \tline{%
	\kern15pt%
	\hbox to 15truemm{\hss(a) $i < j$ or $j=0$\hss}%
	\hfill
	\hbox to 15truemm{\hss(b) $i > j \ge 2$\hss}%
	\hfill
	\hbox to 65truept{\hss(c) $k_1 = q+1$\hss}%
	\hfill
	\hbox to 80truept{\hss(d) $k_1 = q-1$\hss}%
	\kern15pt%
 }
 \begin{center}
 \caption{Cases 2 and 3.1\label{fig-case2}}
 \end{center}
\end{figure}

 \smallskip
 \noindent\textbf{ Case 2.} Suppose that $m=2$ and $e_1$ and $e_2$ are in
different columns.  Without loss of generality suppose $e_1$ is in
column $1$ and $e_2$ is in column $0$.  By applying an appropriate power
of $U$, and possibly $F_2$, we can move $e_1$ so that $e_1 =
\vv(0,0)\ee\vv(1,1)$.  Then $e_2 = \vv(0,i)\ee\vv(1,j)$ where $i$ is odd
and $j$ is even and $0 \le i,j \le n-1$.  There are slightly different
pictures depending on the order of $i$ and $j$.  Figure
\ref{fig-case2}(a) shows the hamilton cycle through $e_1$ and $e_2$ for
$i<j$, and (b) is for $i>j \ge 2$.  The case $i> j = 0$ is treated as
$i < j=n$, using Figure \ref{fig-case2}(a) with the path
 $\ldots\vv(0,i)\ee\vv(1,j=n=0)\vv(1,n-1)\ee\vv(1,n-2)\ldots\vv(1,i+1)\ee%
	\vv(0,i+1)\ee\vv(0,i+2)\ee\vv(0,i+3)\ldots\vv(0,n=0)\ee%
	\vv(1,1)\ldots$.

 \smallskip
 \noindent\textbf{ Case 3.} Suppose that $m=1$.  
 Then $G$ has a single vertical cycle $C=C_n$ containing all vertices,
and we identify vertices with elements of $\mZ_n$.
 We write edges and paths as comma-separated sequences of vertices
inside parentheses, and to indicate a cycle we use double parentheses,
so that for example $C=\cyc{0, 1, 2, \ldots, n-1}$.
 An edge $(i,j)$ with $j-i = \pm k$ (mod $n$), $2 \le k \le n/2$, is
called a \emph{$k$-chord}, or just a \emph{chord}.
 $G$ is a circulant graph containing edges of $C$ and all possible
$q$-chords.  The added diagonals $e_1$ and $e_2$ are $(q \pm 1)$-chords.
 Two chords $(i,j)$ and $(k,\ell)$ \emph{cross} if $i, j, k, \ell$ are
distinct and appear in the order $i, k, j, \ell$, or its reverse, along
$C$.


 Since $G$ is bipartite, $n$ is even and $q$ must be odd.  Since $G$ is
simple, $q \ne 0, 1, -1$ or $n/2$ (mod $n$).  Moreover, $Q(1,n;q)$ is
identical to $Q(1,n;-q=n-q)$ (both embeddings have the same underlying
graphs and facial cycles) and so we may assume that $3 \le q < n/2$. 
Thus, $n \ge 2q+2 \ge 8$.
 We may assume that $e_1$ is a $k_1$-chord and $e_2$ is a $k_2$-chord,
where $k_1 \ge k_2$ and $k_1, k_2 \in \set{q-1, q+1}$.

 For this case it is difficult to use our standard picture of the
embedding on the torus.  With only one column of vertices, the desired
cycle may use many of the edges crossing column $0$, which makes it
difficult to follow.  Thus, for this case we will use two alternative
representations.

 \smallskip
 \noindent\textbf{Case 3.1.} Suppose $e_1$ and $e_2$ cross.
 Using automorphisms of $G$, we may suppose that $e_1 = (0, k_1)$ and
$e_2 = (a, b=a+k_2)$ where $1 \le a \le k_1-1$ and $k_1+1 \le b \le
k_1+k_2-1$.

 In this case we break the cycle $C$ into two segments,
depicted as vertical paths, so straight vertical edges are edges of $C$.
Straight horizontal edges represent $q$-chords $(i,i+q)$ with $i$ at
left, $i+q$ at right.  Other edges must be identified using their
endvertices.
 Quadrangles bounded by horizontal and vertical edges represent faces
in the embedding, although we do not see all faces in our picture.

 If $k_1 = q+1$ then $1 \le a \le q$ and $q+2 \le b \le 2q+1 \le n-1$,
and we have a hamilton cycle through $e_1$ and $e_2$ as shown in Figure
\ref{fig-case2}(c), using either $(a-1,a+q-1,a+q,a,a+q+1,a+1)$ if
$k_2=q+1$, or $(a-1,a+q-1,a,a+q,a+q+1,a+1)$ if $k_2=q-1$.

 If $k_1 = q-1$ then $k_2 = q-1$ also.  Then $1 \le a \le q-2$ and $q
\le b \le 2q-3 < n-1$.  We have the hamilton cycle shown in Figure
\ref{fig-case2}(d).

 \smallskip
 \noindent\textbf{Case 3.2.} Suppose $e_1$ and $e_2$ do not cross.  Note
that $e_1$ and $e_2$ have no common vertex because one is a black-black
diagonal and the other is a white-white diagonal.
 Using automorphisms of $G$, we may suppose that $e_1 = (0, k_1)$ and
$e_2 = (a, b=a+k_2)$ where $a \ge k_1+1$ and $k_1+k_2+1 \le b \le
n-1$.
 Regard all vertices as nonnegative integers $i$ with $0 \le i \le
n-1$, so that we can order them.

 In this case we will draw $C$ as a circle so the other edges are
literally chords of this circle.
 The general pattern is to divide the vertices up into cycles by taking
$q+1$ consecutive vertices along the circle and closing up the cycle
with a chord.  The added diagonals $e_1$ and $e_2$ give cycles of length
$q+2$ or $q$.
 Next these cycles are connected by choosing an edge $f=(i,i+1)$ of $C$
in one cycle and an edge $f'=(j,j+1)$ in the next cycle, so that
$g=(i,j)$ and $g'=(i+1,j+1)$ are $q$-chords, and removing $f$ and $f'$,
and then replacing them by $g$ and $g'$, to merge the two cycles
together. Leftover vertices are incorporated using a similar strategy,
and eventually everything is merged into a single cycle.  Care must be
taken so that edges of a cycle used for one purpose (such as linking to
the previous cycle) do not overlap with those used for another purpose
(such as linking to the next cycle, or to leftover vertices).

 Recall that $q$ is odd, so $k_1, k_2 = q \pm 1$ are even.  Also, $e_1$
is a black-black edge while $e_2$ is a white-white edge, so $a$ is odd. 
 Let $C_1 = \cyc{0, 1, 2, \ldots, k_1}$ and let $C_2 = \cyc{a, a+1, a+2,
\ldots, a+k_2}$.  Consider the vertices along $C$
after $C_1$ but before $C_2$, which we wish to partition into
$(q+1)$-cycles as far as possible.
 For each integer $i$ let $x_i = k_1 + 1 + i(q+1)$ and let $p =
\max\set{i \;|\; x_i \le a}$; then $p \ge 0$.  We have $p$
$(q+1)$-cycles $D_0, D_1, \ldots, D_{p-1}$ where $D_i = \cyc{x_i, x_i+1,
\ldots, x_i+q=x_{i+1}-1}$.
 This leaves vertices $x_p, x_p+1, \ldots, a-1$: since $x_p$ and $a$ are
both odd, there are an even number of these, from which we form a
(possibly empty) matching $M = \set{(x_p, x_p+1), (x_p+2,x_p+3), \ldots,
(a-2,a-1)}$.

 In a similar way we let $y_i = a+k_2+1 + i(q+1)$, $r = \max\set{i \;|\;
y_i \le n} \ge 0$ and divide the vertices along $C$ after $C_2$ but
before $C_1$ into $r$ $(q+1)$-cycles $E_0, E_1, \ldots E_{r-1}$, where
$E_i = \cyc{y_i, y_i+1, \ldots, y_i+q = y_{i+1}-1}$.  Since $y_r$ and
$n$ are both even, there are an even number of leftover vertices from
which we form a (possibly empty) matching $N = \set{(y_r, y_r+1),
(y_r+2, y_r+3), \ldots, (n-2, n-1)}$. So as we go along $C$ the vertices
are partitioned into a sequence of subgraphs $\S = C_1, D_0, D_1,
\ldots, D_{p-1}, M, C_2, E_0, E_1, \ldots, E_{r-1}, N$ (omitting $M$ or
$N$ if they are empty).  We need to merge these into a single hamilton
cycle that uses $e_1$ and $e_2$.

 Given an edge $(i, i+1)$, the edges $(i+q, i+q+1)$ and $(i-q,i-q-1)$
are called its \emph{forward} and \emph{backward mates}, respectively. 
If we have two vertex-disjoint cycles $Z$ containing $(i,i+1)$ and $Z'$
containing its mate $(i+q, i+q-1)$ then we may combine them into a new
cycle $Z \cup Z' - \set{(i,i+1),(i+q,i+q+1)} \cup \set{(i,i+q), (i+1,
i+q+1)}$.  We call this a \emph{cycle-to-cycle link}, or \emph{CC-link}.
 If we have a cycle $Z$ containing $(i,i+1)$, and its forward mate
$(i+q,i+q+1)$ is vertex-disjoint from $Z$ (this mate will be an edge in
one of the matchings $M$ or $N$), then we may combine them into a new
cycle $Z - (i,i+1) \cup (i,i+q,i+q+1,i+1)$.  We may apply a similar
operation using the backward mate $(i-q,i-q+1)$.
 We call this a \emph{cycle-to-edge link}, or \emph{CE-link}.

 Our basic idea is to link together consecutive subgraphs in the
sequence $\S$ using CC- and CE-links.
 An edge of $C$ belonging to a subgraph of $\S$ is
\emph{forward-linking} if its forward mate is in the next subgraph of
$\S$, and \emph{backward-linking} if its backward mate is in the
previous subgraph of $\S$.
 To avoid conflicts between forward- and backward-linking edges we
classify an edge $e=(i,i+1)$ of $C$ as \emph{odd} or \emph{even}
according to whether $i$, its smaller end, is odd or even, respectively.
 Note that a mate of $e$ is odd when $e$ is even, and vice versa,
because $q$ is odd.
 In each cycle of $\S$ we will use odd edges to link in one direction
and even edges to link in the opposite direction.

 Suppose we have two consecutive cycles $Z, Z'$ in $\S$.  We may write
$Z = \cyc{i-s, i-s+1, \ldots, i}$ and $Z'=\cyc{i+1, i+2, \ldots, i+t+1}$
where $s, t \in \set{q-1, q, q+1}$.
 Because $q \ge 3$ and $s, t \ge q-1$, $Z \cap C$ always contains the
two edges $(i-2,i-1)$ and $(i-1,i)$, and they are always forward-linking
because $Z'$ always contains their forward mates $(i+q-2,i+q-1)$ and
$(i+q-1,i+q)$.
 Therefore we always have both an odd forward-linking edge of
$Z$ mated with an even backward-linking edge of $Z'$, and an even
forward-linking edge of $Z$ mated with an odd backward-linking edge of
$Z'$.

 Suppose we have a matching $L$ preceded by a cycle $Z$ in $\S$. We may
write $Z = \cyc{i-s, i-s+1, \ldots, i}$ and $L = \set{(i+1, i+2), (i+3,
i+4), \ldots, (i+t-1, i+t)}$ where $s \in \set{q-1, q, q+1}$ and $t$ is
even with $2 \le t \le q-1$.  Now if $(i+j, i+j+1)$ is an edge of $L$
then $1 \le j \le t-1 \le q-2$, so that $i-s \le i+1-q \le i+j-q$ and
$i+j+1-q \le i-1$, which shows that the backward mate $(i+j-q,i+j+1-q)$
is in $Z$.  Thus, every edge of $L$ is backward-linking.  Similarly, if
a matching $L$ is followed by a cycle in $\S$, then every edge of $L$ is
forward-linking.

\begin{figure}
 \tline{%
	\kern15pt%
	\epsfbox{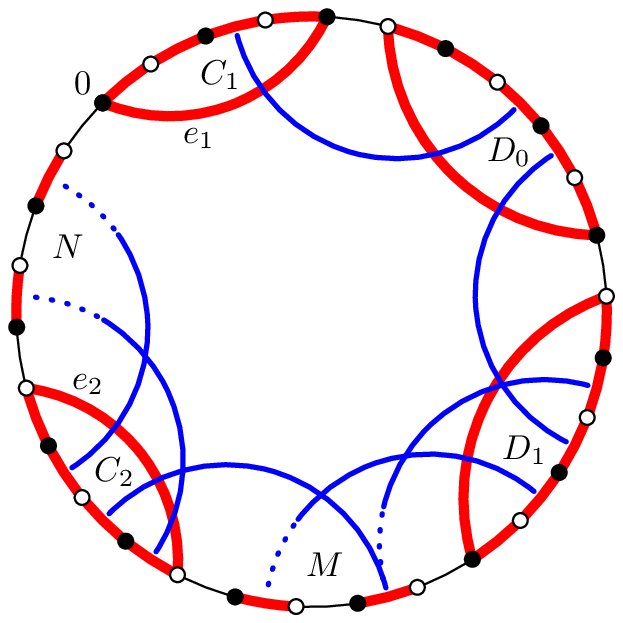}%
	\hfill
	\epsfbox{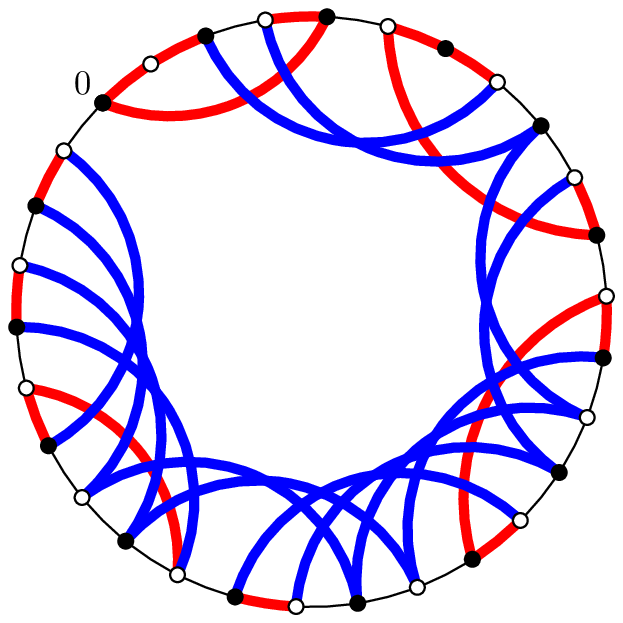}%
	\kern15pt%
 }
 \tline{%
	\kern15pt%
	\hbox to 60truemm{\hss(a) Schematic of links\hss}%
	\hfill
	\hbox to 60truemm{\hss(b) Hamilton cycle $H$\hss}%
	\kern15pt%
 }
 \begin{center}
 \caption{Case 3.2.2\label{fig-case3.2.2}}
 \end{center}
\end{figure}

 \smallskip
 \noindent\textbf{Case 3.2.1.} Suppose $M = \emptyset$.  Use odd
forward-linking edges and even backward-linking edges and repeated
CC-linking to combine all of $C_1, D_0, D_1, \ldots, D_{p-1}, C_2, E_0,
\ldots, E_{r-1}$ into a single cycle $Z_1$.  Then $Z_1$ still contains
all odd edges of the last cycle ($E_{r-1}$, or $C_2$ if $r=0$) so these
can be used to incorporate all edges of $N$ (if any), which are even, by
repeated CE-linking to give the final hamilton cycle $H$.
 Since we delete only edges of $C$ when linking, $H$ contains the chords
$e_1$ from $C_1$ and $e_2$ from $C_2$, as required.

 \smallskip
 \noindent\textbf{Case 3.2.2.} Suppose $M \ne \emptyset$.
 We form a cycle $H_2$ containing all vertices of $C_2, E_0, E_1,
\ldots, E_{r-1}$ and $N$ as in Case 3.2.1.  Note that $H_2$ contains all
even edges of $C_2$.  In a similar way, but switching the roles of odd
and even edges, we form a cycle $H_1$ containing all vertices of $C_1,
D_0, D_1, \ldots, D_{p-1}$ and $M$.  $H_1$ contains all edges of $M$,
which are odd.  We can now CC-link
$H_1$ and $H_2$ using the first edge of $M$, $(x_p,x_p+1)$, and its even
forward mate in $C_2$, to form a hamilton cycle $H$.
  As before, $H$ contains $e_1$ and $e_2$.

 This process is illustrated in Figure \ref{fig-case3.2.2}, where we
have $n = 30$, $q=5$, $k_1=4$, $p=2$, $k_2=4$, $r=0$ and $|M|=|N|=2$. 
In (a) we show a schematic of where the links are added: CC-links are
given by solid lines, and CE-links by lines that are dashed at the
matching end (to indicate that the matching edge is not deleted).  In
(b) we show the corresponding hamilton cycle $H$.  No edge of $C$ is
used by two CC-links, but the first edge of $M$, $(x_p, x_p+1)$, is used
by both a CE-link and a CC-link.

 \smallskip
 This concludes the proof of Proposition \ref{prop-main}.
 \end{proof}

 Now we prove our main result, which we restate.

 \begin{theoremmain}
 \MainTheoremText
 \end{theoremmain}

 \begin{proof}
 Let $e$ be an edge of $G' = G \cup E_1 \cup \set{e_2}$.  Suppose first
that $e \in E(G)$.  We use the notation developed in the proof of
Proposition \ref{prop-main}.  If $m \ge 2$ then, since $n$ is even, it
is easy to construct a hamilton cycle in $G$ consisting of a vertical
path with ends joined by a horizontal zigzag, which uses at least one
vertical and one horizontal edge.  Since $e$ is either vertical or
horizontal (including edges across column $0$), and all vertical edges
are similar in $G$ and all horizontal edges are similar in $G$, we can
use an automorphism of $G$ to find a hamilton cycle of $G$, and hence of
$G'$, through $e$.
 If $m=1$ then, using the notation from Case 3 of the above proposition,
$G$ has a hamilton cycle
 $\cyc{0, q, q-1, q-2, \ldots, 2, 1, q+1, q+2, q+3, \ldots, n-1}$
 which uses both vertical edges (edges of $C$) and horizontal edges
($q$-chords).  Again, $e$ is either vertical or horizontal, and using an
automorphism of $G$ we can find a hamilton cycle through $e$.

 So suppose $e \in E_1$, or $e = e_2$.  If $e \in E_1$ we let $e' =
e_2$, and if $e = e_2$ we choose any $e' \in E_1$.  By Proposition
\ref{prop-main} there is a hamilton cycle through $e$ and $e'$ in $G
\cup \set{e,e'}$ and hence in $G'$.
 \end{proof}

 \section{Conclusion}\label{sec-concl}

 Our results provide some evidence that bipartiteness is the underlying
factor preventing $4$-connected toroidal graphs from being
edge-hamiltonian.  Unfortunately, we had to restrict ourselves to
examining graphs derived from the grid-type examples.  For other
examples, such as that shown in Figure \ref{fig-nongrid}, we do not have
a good structure theorem, and it is difficult even to know if a graph
constructed by adding diagonals to a bipartite quadrangulation of the
torus is $4$-connected.

 However, something at least is known about bipartite quadrangulations
of the torus.
 The graphs we are interested in are bipartite quadrangulations that can
yield a $4$-connected graph with the addition of diagonals on one side
of the bipartition (say, black-black diagonals).  It is not difficult to
show that this can happen only if all white vertices have degree exactly
$4$.
 Fujisawa, Nakamoto and Ozeki  \cite{FuNaOz13} recently showed that
bipartite quadrangulations of the torus in which all white vertices have
degree $4$ are hamiltonian, satisfying Conjecture \ref{conj-grnw}, as long as
they are at least $3$-connected.  Perhaps their techniques may yield
some results on edge-hamiltonicity after diagonals are added.

 \smallskip
 There is also a similar conjecture to Conjecture \ref{conj-grnw} for
the Klein bottle, and similar counterexamples to edge-hamiltonicity,
based on $4$-connected bipartite quadrangulations of the Klein bottle. 
A characterization of such quadrangulations is known \cite{NaNe00,
Th91}, but it is significantly more complicated than for the torus, and
the quadrangulations themselves are not as symmetric as those on the
torus, meaning that many more cases would have to be examined to obtain
a result similar to Theorem \ref{thm-main}.

 \section*{Acknowledgements}

 The first author acknowledges support from the U.S. National Security
Agency (NSA) under grant number H98230--09--1--0065.
 Both authors acknowledge support from the NSA under grant number
H98230--13--1--0233, and from the Simons Foundation under award number
245715.


\begin{thebibliography}{99}
 \def\itj#1{{\it\frenchspacing #1\/}} 


 \bibitem{Al72}
 Amos Altshuler, Hamiltonian circuits in some maps on the torus,
\itj{Discrete Math.} \textbf{1} (1972) 299--314.


 \bibitem{BrRi95}
 R. Brunet and R.B. Richter, Hamiltonicity of 5-connected toroidal
triangulations, \itj{J. Graph Theory} \textbf{20} (1995) 267--286.

 \bibitem{ChNi86}
 Norishige Chiba and Takao Nishizeki, A theorem on paths in planar
graphs, \itj{J. Graph Theory} \textbf{10} (1986) 449--450.

 \bibitem{FuNaOz13}
 Jun Fujisawa, Atsuhiro Nakamoto and Kenta Ozeki,
 Hamiltonian cycles in bipartite toroidal graphs with a partite set of
degree four vertices, \itj{J. Combin. Theory Ser. B} \textbf{103} (2013)
46--60.

 \bibitem{Gr70}
 Branko Gr\"unbaum, Polytopes, graphs and complexes, \itj{Bull. Amer.
Math. Soc.} \textbf{76} (1970) 1131--1201.

\bibitem{NaNe00}
 Atsuhiro Nakamoto and Seiya Negami, Full-symmetric embeddings of
graphs on closed surfaces, \itj{Mem. Osaka Kyoiku Univ. Ser. III}
\textbf{49} (2000) 1-15.

 \bibitem{NaOz12}
 Atsuhiro Nakamoto and Kenta Ozeki,
 Hamiltonian cycles in bipartite quadrangulations on the torus, \itj{J.
Graph Theory} \textbf{69} (2012) 143--151.

 \bibitem{NW73}
 C.~St.~J.~A. Nash-Williams, Unexplored and semi-explored territories in
graph theory, in ``New Directions in the Theory of Graphs,'' edited by
Frank Harary, (Proc. Third Ann Arbor Conf., Univ. Michigan, Ann Arbor,
Mich., 1971), Academic Press, New York (1973) 149--186.

 \bibitem{ThYu94}
 Robin Thomas and Xingxing Yu, $4$-connected projective-planar graphs
are hamiltonian, \itj{J. Combin. Theory Ser. B} \textbf{62} (1994)
114--132.

 \bibitem{ThYu97}
 Robin Thomas and Xingxing Yu, 5-connected toroidal graphs are
hamiltonian, \itj{J. Combin. Theory Ser. B} \textbf{69} (1997) 79--96.

 \bibitem{ThYuZa05}
 R. Thomas, X. Yu and W. Zang, Hamilton paths in toroidal graphs,
\itj{J. Combin. Theory Ser. B} \textbf{94} (2005) 214--236.

 \bibitem{Th83}
 Carsten Thomassen, A theorem on paths in planar graphs, \itj{J. Graph
Theory} \textbf{7} (1983) 169--176.

 \bibitem{Th91}
 Carsten Thomassen, Tilings of the torus and the Klein bottle and
vertex-transitive graphs on a fixed surface, \itj{Trans. Amer. Math.
Soc.} \textbf{323} (1991) 605--635.

 \bibitem{Tu56}
 W. T. Tutte, A theorem on planar graphs, \itj{Trans. Amer. Math. Soc.}
\textbf{82} (1956) 99--116.

 \bibitem{Tu77}
 W. T. Tutte, Bridges and hamiltonian circuits in planar graphs,
\itj{Aequationes Math.} \textbf{15} (1977) 1--33.

 \bibitem{Wh31}
 H. Whitney, A theorem on graphs, \itj{Ann. Math.} \textbf{32} (1931)
378--390.

 \end{thebibliography}
 \end{document}